\documentclass[11pt]{amsart}

\newtheorem{theorem}{Theorem}[section]
\newtheorem{lemma}[theorem]{Lemma}

\theoremstyle{definition}
\newtheorem{definition}[theorem]{Definition}

\newtheorem{corollary}[theorem]{Corollary}
\newtheorem{remark}[theorem]{Remark}

\theoremstyle{remark}

\newcommand{\be}{\begin{equation}}
\newcommand{\ee}{\end{equation}}

\numberwithin{equation}{section}

\begin{document}

\title{Circle action and some vanishing results on manifolds}

\author{Ping Li}
\address{Department of Mathematics, Tongji University, Shanghai 200092, China}
\email{pingli@tongji.edu.cn}
\thanks{The first author is supported by Program for Young Excellent
Talents in Tongji University.}

\author{Kefeng Liu} \address{Department of Mathematics, University of California at Los Angeles, Los Angeles, CA 90095, USA and Center of Mathematical Science,
Zhejiang University, 310027, China} \email{liu@math.ucla.edu}




\begin{abstract}
Kawakubo and Uchida showed that, if a closed oriented
$4k$-dimensional manifold $M$ admits a semi-free circle action such
that the dimension of the fixed point set is less than $2k$, then
the signature of $M$ vanishes. In this note, by using $G$-signature
theorem and the rigidity of the signature operator, we generalize
this result to more general circle actions. Combining the same idea
with the remarkable Witten-Taubes-Bott rigidity theorem, we explore
more vanishing results on spin manifolds admitting such circle
actions. Our results are closely related to some earlier results of
Conner-Floyd, Landweber-Stong and Hirzebruch-Slodowy.
\end{abstract}

\maketitle

\section{Introduction and results}
Unless otherwise stated, all the manifolds discussed in this paper
are closed smooth manifolds and all involutions and circle actions
on the manifolds are smooth. We denote by superscripts the
corresponding dimensions of the manifolds.

The following is a classical result of Conner and Floyd (\cite{CF},
$\S 27.2$).

\begin{theorem}[Conner-Floyd]\label{CF}
Suppose $g:~M^{2n}\rightarrow M^{2 n}$ is an involution on a
manifold and $M^{g}$ is the fixed point set of $g$. If
$\textrm{dim}(M^{g})< n$, then the Euler characteristic of $M$ is
even.
\end{theorem}
Here by $\textrm{dim}(M^{g})$ we mean the dimension of highest
dimensional connected component of $M^{g}$.

Using their famous $G$-signature theorem, Atiyah and Singer reproved
(\cite{AS}, p.582-p.583) Theorem \ref{CF} when $n$ is even, $M$ is
oriented and $g$ is orientation preserving.

We recall that a circle action ($S^{1}$-action) is called
\emph{semi-free} if it is free on the complement of the fixed point
set or equivalently, the isotropy subgroup of any non fixed point on
the manifold is trivial. Using bordism techniques developed by
Conner and Floyd in \cite{CF}, Kawakubo and Uchida showed the
following result (\cite{KU}, Theorem 1.2), which could be taken as a
counterpart in the circle case to Theorem \ref{CF} in some sense.

\begin{theorem}[Kawakubo-Uchida]\label{KU}
Suppose $M^{4k}$ admits a semi-free $S^{1}$-action and $M^{S^{1}}$
is the fixed point set of this action. If $\textrm{dim}(M^{S^{1}})<
2k$, then the signature of $M$, $\textrm{sign}(M^{4k})$, is zero.
\end{theorem}

Our first purpose in this note is, by closer looking at the
$G$-signature theorem in the circle case, to generalize Theorem
\ref{KU} to more general cases. Before stating our first main
result, we will introduce some notations, which will be used
throughout this paper without further explanation.

Suppose $M^{2n}$ is a oriented manifold admitting a $S^{1}$-action.
Let $F^{2m}$ be a connected component of this action. With respect
to this $S^{1}$-action, the tangent bundle of $M^{2n}$ restricted to
$F^{2m}$, $\textrm{T}M^{2n}\big|_{F^{2m}}$, has the following
equivariant decomposition:

$$\textrm{T}M^{2n}\big|_{F^{2m}}=L_{1}\oplus\cdots\oplus L_{n-m}\oplus\textrm{T}F^{2m},$$
where each $L_{i}$ is a real $2$-plane bundle of $F^{2m}$. We can
identify $L_{i}$ with a complex line bundle relative to which the
representation of $S^{1}$ on each fiber of $L_{i}$ is given by
$e^{\sqrt{-1}\theta}\rightarrow e^{\sqrt{-1}k_{i}\theta}$ with
$k_{i}\in \mathbb{Z}-\{0\}$. These $k_{1},\cdots,k_{n-m}$ are called
\emph{weights} of this $S^{1}$-action on the connected component
$F^{2m}$ and uniquely determined up to
signs.

\begin{definition}\label{def}
Let the notations be as above. We call a $S^{1}$-action \emph{prime}
if there exists a number $\xi\in S^{1}$ such that, for any $k\in
\bigcup_{F^{2m}}\{k_{1},\cdots,k_{n-m}\},$  we have $\xi^{k}=-1$.
\end{definition}
\begin{remark}
Note that the weights of a semi-free circle action are $\pm 1$.
Hence semi-free circle actions are prime.
\end{remark}
Now we can state our first result, which generalizes Theorem
\ref{KU}.

\begin{theorem}\label{result1}
Suppose $M^{4k}$ admits a prime $S^{1}$-action and $M^{S^{1}}$ is
the fixed point set of this action. If $\textrm{dim}(M^{S^{1}})<
2k$, then $\textrm{sign}(M^{4k})=0$.
\end{theorem}

We will prove this result in Section 2. The signature of a oriented
manifold can be realized as an index of some elliptic operator
(\cite{AS}, \S 6), now called signature operator. Besides the
$G$-signature theorem, the key ingredient of the proof of Theorem
\ref{result1} is the rigidity of the signature operator (see Section
2 for more details). The rigidity of signature operator is only the
beginning of a remarkable rigidity theorem: Witten-Taubes-Bott
rigidity theorem. Our second purpose in this note is, by using this
rigidity theorem, to replace the conclusion of $\textrm{sign}(M)=0$
in Theorem \ref{result1} by those of vanishing indices of some
twisted signature operators.

In order to state our second result, let us begin with the rigidity
of elliptic operators.

Let $D:~\Gamma(E)\rightarrow\Gamma(F)$ be an elliptic operator
acting on sections of complex vector bundles $E$ and $F$ over a
manifold $M$. Ellipticity guarantees that both $\textrm{ker}(D)$ and
$\textrm{coker}(D)$ are finite-dimensional. Then the index of $D$ is
defined as
$$\textrm{ind}(D)=\textrm{dim}_{\mathbb{C}}\textrm{ker}(D)-\textrm{dim}_{\mathbb{C}}\textrm{coker}(D).$$
If $M$ admits an $S^{1}$-action preserving $D$, i.e., acting on $E$
and $F$ and commuting with $D$, then both $\textrm{ker}(D)$ and
$\textrm{coker}(D)$ admit an $S^{1}$-action and hence are
$S^{1}$-modules. Therefore the virtual complex vector space
$\textrm{ker}(D)-\textrm{coker}(D)$ has a Fourier decomposition into
a finite sum of complex one-dimensional irreducible representations
of $S^{1}$:
$$\textrm{ker}(D)-\textrm{coker}(D)=\sum a_{i}\cdot L^{i},$$
where $a_{i}\in\mathbb{Z}$ is the representation of $S^{1}$ on
$\mathbb{C}$ given by $\lambda\mapsto\lambda^{i}$. The equivariant
index of $D$ at $g\in S^{1}$, $\textrm{ind}(g, D)$, is defined to be
$$\textrm{ind}(g, D)=\sum a_{i}\cdot g^{i}.$$
The elliptic operator $D$ is called \emph{rigid} with respect to
this $S^{1}$-action if $a_{i}=0$ for all $i\neq 0$, i.e.,
$\textrm{ker}(D)-\textrm{coker}(D)$ consists of the trivial
representation with multiplicity $a_{0}$. Consequently,
$\textrm{ind}(g, D)\equiv\textrm{ind}(D)$ for any $g\in S^{1}$. An
elliptic operator is called \emph{universally rigid} if it is rigid
with respect to \emph{any} $S^{1}$-action. The fundamental examples
of universally rigid operators are signature operator and Dirac
operator (on spin manifolds). The reason for the former is that both
of its kernel and cokernel can be identified with subspaces of the
deRham cohomology group (\cite{AS}, $\S 6$) on which $S^{1}$ always
induces a trivial action. The latter is a classical result of Atiyah
and Hirzebruch \cite{AH}.

Let $\Omega_{\mathbb{C}}^{+}$ and $\Omega_{\mathbb{C}}^{-}$ be the
even and odd complex differential forms on a oriented Riemann
manifold $M^{2n}$ under the Hodge $\ast$-operator. Then the
signature operator
$$d_{s}:~\Omega_{\mathbb{C}}^{+}\rightarrow\Omega_{\mathbb{C}}^{-}$$
is elliptic and the index of $d_{s}$ equals to $\textrm{sign}(M)$
(\cite{AS}, \S 6).

Let $W$ be a complex vector bundle over $M$. By means of a
connection on $W$, the signature operator can be extended to a
twisted operator (\cite{Pa}, IV, $\S 9$)
$$d_{s}\otimes W:~\Omega_{\mathbb{C}}^{+}(W)\rightarrow\Omega_{\mathbb{C}}^{-}(W).$$
This operator is also elliptic and the index of $d_{s}\otimes W$ is
denoted by $\textrm{sign}(M, W)$.

Let $T_{\mathbb{C}}$ be the complexified tangent bundle of $M$. For
an indeterminate $t$, set
$$\Lambda_{t}T_{\mathbb{C}}=\sum_{k=0}^{\infty}t^{k}\Lambda^{k}T_{\mathbb{C}},\qquad S_{t}T_{\mathbb{C}}=\sum_{k=0}^{\infty}t^{k}S^{k}T_{\mathbb{C}},$$
where $\Lambda^{k}T_{\mathbb{C}}$ and $S^{k}T_{\mathbb{C}}$ are the
$k$-th exterior power and symmetry power of $T_{\mathbb{C}}$
respectively (\cite{At}, $\S
3.1$).

Let $R_{i}$ be the sequence of bundles defined by the formal
series
$$\sum_{i=0}^{+\infty}q^{i}R_{i}=\bigotimes_{i=1}^{+\infty}\Lambda_{q^{i}}T_{\mathbb{C}}\otimes\bigotimes_{j=1}^{+\infty}S_{q^{j}}T_{\mathbb{C}}.$$
The first few terms of this sequence
are
$$R_{0}=1,\qquad R_{1}=2T_{\mathbb{C}},\qquad R_{2}=2(T_{\mathbb{C}}\otimes T_{\mathbb{C}}+T_{\mathbb{C}}),\qquad\cdots$$
With all this understood we have the following rigidity theorem.
\begin{theorem}[Witten-Taubes-Bott]\label{WTB}
For a spin manifold $M^{2n}$, each of the elliptic operators
$d_{s}\otimes R_{i}$ is universally rigid.
\end{theorem}
This rigidity theorem was conjectured  and given a string-theoretic
interpretation by Witten \cite{Wi2}. It was first proved by Taubes
\cite{Ta}. A simper proof was then presented by Bott and Taubes
\cite{BT}. Using modular invariance of Jacobi functions, the second
author gave a more simpler and unified new proof in \cite{Li1} and
further generalized them in \cite{Li2}.

Using this remarkable rigidity theorem, Hirzebruch and Slodowy
showed that (\cite{HS}, p.317), among other things, if $g$ is an
involution contained in a circle acting on a spin manifold $M^{4k}$
and $\textrm{dim}(M^{g})< 2k,$ then $\textrm{sign}(M, R_{i})=0$ for
all $i$.

We are now ready to state our second main result, which could be
taken as the counterpart to Hirzebruch-Slodowy's above mentioned
result in the circle case.

\begin{theorem}\label{result2}
Suppose $M^{2n}$ is a spin manifold admitting a prime
$S^{1}$-action. If $\textrm{dim}(M^{S^{1}})< n$, then
$\textrm{sign}(M^{2n}, R_{i})=0$ for all $i$.
\end{theorem}

\begin{corollary}
Suppose $M^{2n}$ is a spin manifold admitting a semi-free
$S^{1}$-action. If $\textrm{dim}(M^{S^{1}})< n$, then
$\textrm{sign}(M^{2n}, R_{i})=0$ for all $i$.
\end{corollary}

\begin{remark}
In \cite{LS}, Landweber and Stong proved two results concerning the
signature and the indices of three twisted Dirac operators, which
also have the same feature as our results in some sense. More
precisely, they showed that (\cite{LS}, Theorem 1), if a closed spin
manifold $M^{2n}$ admits a circle action of \emph{odd type}, then
$\textrm{sign}(M)=0$. Moreover, if this action is semi-free, then
(\cite{LS}, Theorem 2) $\hat{A}(M, T_{\mathbb{C}})=\hat{A}(M,
\Lambda^{2}T_{\mathbb{C}})=\hat{A}(M,
\Lambda^{3}T_{\mathbb{C}}+T_{\mathbb{C}}^{2})=0$, where $\hat{A}(M,
E)$ is the index of the Dirac operator on $M$ twisted by a complex
vector bundle $E$.
\end{remark}

\section{Proof of results}
Let $M^{2n}$ be a oriented manifold admitting an $S^{1}$-action. Let
$F^{2m}$ be a connected component of the fixed point set
$M^{s^{1}}$. As pointed out in Introduction,
$\textrm{T}M^{2n}\big|_{F^{2m}}$ can be decomposed into
$$\textrm{T}M^{2n}\big|_{F^{2m}}=L_{1}\oplus\cdots\oplus L_{n-m}\oplus\textrm{T}F^{2m}.$$
Here $L_{i}$ could be taken as a complex line bundle over $F^{2m}$
with weight $k_{i}$, $1\leq i\leq n-m$.

$F^{2m}$ can be oriented so that all orientations of $L_{1},\cdots,
L_{n-m}$ and $F^{2m}$ taken together yield the orientation of
$M^{2n}$. Let $c_{1}(L_{i})\in H^{2}(F^{2m};\mathbb{Z})$ be the
first Chern class of $L_{i}$. Suppose the total Pontrjagin class of
$F^{2m}$ has the following formal
decomposition
$$p(F^{2m})=1+p_{1}(F^{2m})+\cdots=\prod_{i=1}^{m}(1+x_{i}^{2}),$$
i.e., $p_{i}(F^{2m})$ is the $i$-th elementary symmetry polynomial
of $x_{1}^{2},\cdots, x_{m}^{2}$.

With these notations set up, we have the following important lemma,
which should be well-known for experts (cf. \cite{HBJ}, $\S 5.8$),
although, according to the authors acknowledge, nobody state it
explicitly as follows.
\begin{lemma}\label{lemma}
Let $g$ be an indeterminate. Then the rational function of $g$
\be\sum_{F^{2m}}\big\{\big[\big(\prod_{i=1}^{m}x_{i}\frac{1+e^{-x_{i}}}
{1-e^{-x_{i}}}\big)\big(\prod_{j=1}^{n-m}\frac{1+g^{k_{j}}e^{-c_{1}(L_{j})}}{1-g^{k_{j}}e^{-c_{1}(L_{j})}}\big)\big]\cdot[F^{2m}]\big\}\nonumber\ee
identically equals to $\textrm{sign}(M)$. Here $[F^{2m}]$ is the
fundamental class of $F^{2m}$ determined by the orientation and the
sum is over all the connected components of $M^{S^{1}}$.
\end{lemma}
\begin{proof}
Let $g\in S^{1}$ be a topological generator. Then the fixed point
set of the action of $g$ on $M$ are exactly $M^{s^{1}}$. So the
$G$-signature theorem (\cite{AS}, p.582) tells us that
\be\label{signature formula}\textrm{sign}(g,
M^{2n})=\sum_{F^{2m}}\big\{\big[\big(\prod_{i=1}^{m}x_{i}\frac{1+e^{-x_{i}}}
{1-e^{-x_{i}}}\big)\big(\prod_{j=1}^{n-m}\frac{1+g^{k_{j}}e^{-c_{1}(L_{j})}}{1-g^{k_{j}}e^{-c_{1}(L_{j})}}\big)\big]\cdot[F^{2m}]\big\}\ee

Here $\textrm{sign}(g, M^{2n})$ is the equivariant index of the
signature operator at $g\in S^{1}$. According to the rigidity of the
signature operator, we have $\textrm{sign}(g,
M^{2n})\equiv\textrm{sign}(M^{2n})$. Therefore (\ref{signature
formula}) holds for a dense subset of $S^{1}$ (the topological
generators are dense in $S^{1}$), which means (\ref{signature
formula}) is in fact an identity for an indeterminate $g$.
\end{proof}

\begin{remark}
\begin{enumerate}
\item
Lemma \ref{lemma} was used in (\cite{HBJ}, $\S 5.8$), by putting
$g=0$, to obtain the famous formula
$$\textrm{sign}(M^{2n})=\sum_{F^{2m}}\textrm{sign}(F^{2m}),$$
 of which several proofs are given by Atiyah-Hirzebruch (\cite{AH}, $\S 3$), Hattori-Taniguchi (\cite{HT}, $\S 4$),
 and Witten (\cite{Wi1}, $\S 3$) respectively.

\item
When $M^{S^{1}}$ consists of isolated points, Lemma \ref{lemma} was
used by Ding (\cite{Di}, p.3947) to obtain some interesting results
concerning the representations on the isolated fixed points.
\end{enumerate} \end{remark}

\emph{Proof of Theorem \ref{result1}}.
\begin{proof}
Now suppose $M^{4k}$ has a prime $S^{1}$-action. Let $\xi\in S^{1}$
be the desired element as in Definition \ref{def}. Then we have
\be\label{proof}\begin{split}\big[\prod_{j=1}^{2k-m}\frac{1+g^{k_{j}}e^{-c_{1}(L_{j})}}{1-g^{k_{j}}e^{-c_{1}(L_{j})}}\big]\big|_{g=\xi}&=\prod_{j=1}^{2k-m}\frac{1-e^{-c_{1}(L_{j})}}{1+e^{-c_{1}(L_{j})}}\\
&=\prod_{j=1}^{2k-m}\frac{c_{1}(L_{j})-\frac{1}{2}c_{1}^{2}(L_{j})+\cdots}{2-c_{1}(L_{j})+\cdots}\\
&=\big(\prod_{j=1}^{2k-m}c_{1}(L_{j})\big)\cdot\prod_{j=1}^{2k-m}\frac{1-\frac{1}{2}c_{1}(L_{j})+\cdots}{2-c_{1}(L_{j})+\cdots}\\
&=e(\nu
F^{2m})\cdot\prod_{j=1}^{2k-m}\frac{1-\frac{1}{2}c_{1}(L_{j})+\cdots}{2-c_{1}(L_{j})+\cdots},
\end{split}\ee
where $e(\nu F^{2m})\in H^{4k-2m}(F^{2m};\mathbb{Z})$ is the Euler
class of the normal bundle of $F^{2m}$ in
$M^{4k}$.

If $\textrm{dim}(M^{S^{1}})< 2k$, then $4k-2m>2m,$ which means
$e(\nu F^{2m})=0$ and so by Lemma \ref{lemma}
$\textrm{sign}(M^{4k})=0$. This completes the proof.
\end{proof}

\emph{Proof of Theorem \ref{result2}}.
\begin{proof}
Let $R_{i}$ be the complex vector bundles defined in Introduction.
Then for each topological generator $g\in S^{1}$, the equivariant
index of the elliptic operator $d_{s}\otimes R_{i}$,
$\textrm{sign}(g, M^{2n}, R_{i})$, like $G$-signature theorem, could
be computed in terms of the local invariants of the fixed point set
$M^{S^{1}}$. This is given by a general Lefschetz fixed point
formula of Atiyah-Bott-Segal-Singer (\cite{AS}, p.254-p.258).
Instead of writing down the general form of this formula, we only
indicate that, for $\textrm{sign}(g, M^{2n}, R_{i})$, this formula
is of the following form.
$$\sum^{+\infty}_{i=0}q^{i}\cdot\textrm{sign}(g, M^{2n}, R_{i})=\sum_{F^{2m}}\big\{\prod_{i=1}^{m}\big[(x_{i}\frac{1+e^{-x_{i}}}
{1-e^{-x_{i}}}\big)\cdot
u_{i}\big]\prod_{j=1}^{n-m}\big[\big(\frac{1+g^{k_{j}}e^{-c_{1}(L_{j})}}{1-g^{k_{j}}e^{-c_{1}(L_{j})}}\big)\cdot
v_{j}\big]\big\}\cdot[F^{2m}],$$

where
$$u_{i}=\prod_{r=1}^{+\infty}\frac{(1+q^{r}e^{-x_{i}})(1+q^{r}e^{x_{i}})}{(1-q^{r}e^{-x_{i}})(1-q^{r}e^{x_{i}})},$$
and
$$v_{j}=\prod_{r=1}^{+\infty}\frac{(1+q^{r}g^{k_{j}}e^{-c_{1}(L_{j})})(1+q^{r}g^{-k_{j}}e^{c_{1}(L_{j})})}{(1-q^{r}g^{k_{j}}e^{-c_{1}(L_{j})})(1-q^{r}g^{-k_{j}}e^{c_{1}(L_{j})})}.$$
 We recommend the readers the references \big((\cite{Li1},
\S 1 and \S 5) or (\cite{DJ}, $\S 2.4$)\big) for a detailed
description of $\textrm{sign}(g, M^{2n}, R_{i})$ in terms of the
local data of $M^{S^{1}}$.

Rigidity theorem \ref{WTB} says that, for an indeterminate $g$, the
following identity holds
$$\sum^{+\infty}_{i=0}q^{i}\cdot\textrm{sign}(M^{2n}, R_{i})\equiv\sum_{F^{2m}}\big\{\prod_{i=1}^{m}\big[(x_{i}\frac{1+e^{-x_{i}}}
{1-e^{-x_{i}}}\big)\cdot
u_{i}\big]\prod_{j=1}^{n-m}\big[\big(\frac{1+g^{k_{j}}e^{-c_{1}(L_{j})}}{1-g^{k_{j}}e^{-c_{1}(L_{j})}}\big)\cdot
v_{j}\big]\big\}\cdot[F^{2m}],$$
 By using the same idea as in the proof of
Theorem \ref{result1}, we can get the conclusion of Theorem
\ref{result2}.
\end{proof}

\section{Conclusion remarks}
As we have seen, the key idea in the proofs is to extract a
cohomology class $e(\nu F^{2m})$ from the right-hand side of
(\ref{signature formula}), by giving a special value on $g$. In
fact, for a general compact Lie group $G$ acting on $M^{2n}$ and
$g\in G$, the $G$-signature theorem is of the following form
(\cite{AS}, p.582) \be\label{G-s}\textrm{sign}(g,M)=\sum_{F\in
M^{g}}[e\big(N^{g}(-1)\big)\cdot u]\cdot[F],\ee where $F$ is a
connected component of the fixed point set of $g$, $M^{g}$ (rather
than the fixed point set of the whole $G$), $e\big(N^{g}(-1)\big)$
is the Euler class of a subbundle of the normal bundle on $M^{g}$,
corresponding to the eigenvalue $-1$ of the representation of $g$ on
the normal bundle of $F$, and $u\in H^{\ast}(F)$.

Consequently, the right-hand side of (\ref{G-s}) vanishes if, for
every $F$, the fiber dimension of $N^{g}(-1)$ is greater than the
dimension of $F$. This is the corollary $6.13$ in \cite{AS} on page
582. It is this corollary that makes Atiyah and Singer to reprove
Theorem \ref{CF} in some cases, because for an involution, $-1$ is
the \emph{only} eigenvalue on the normal bundle of the fixed point
set. While in the circle case, for a topological generator
$g=e^{2\pi\sqrt{-1}\theta}\in S^{1}$, $-1$ in \emph{not}
 the eigenvalue ($g$ is a topological generator if and only of $\theta$ is irrational,
then $g^{k}\neq -1$ for all weights $k$). So we have to construct
 a cohomology class \big($e(\nu F^{2m})$ in (\ref{proof})\big) similar to $e\big(N^{g}(-1)\big)$ in (\ref{G-s}) artificially.
 This is the origin of our Definition \ref{def}.

\end{document}